\documentclass[a4paper,leqno]{article}

\usepackage{amsmath,amssymb,amsthm}
\usepackage[matrix,arrow,graph,curve,frame,tips]{xy} \SelectTips{cm}{}
\usepackage{rotating}
\theoremstyle{plain}      
    \newtheorem{theorem}{Theorem}[section]
    \newtheorem{proposition}[theorem]{Proposition}
    \newtheorem{lemma}[theorem]{Lemma}
    
\theoremstyle{definition}
    \newtheorem{definition}[theorem]{Definition}
    \newtheorem{example}[theorem]{Example}
    
\theoremstyle{remark}

\newcommand{\C}{\ensuremath{\mathcal{C}}}

\DeclareMathOperator{\Chu}{Chu}

\newcommand{\op}{\ensuremath{\mathrm{op}}}
\newcommand{\ox}{\ensuremath{\otimes}}
\newcommand{\dia}{\ensuremath{\diamond}}

\newcommand{\wt}{\ensuremath{\widetilde}}

\newcommand{\ra}{\ensuremath{\rightarrow}}

\newcommand{\xra}{\ensuremath{\xrightarrow}}

\begin{document}
\title{Note on star-autonomous comonads}
\author{Craig Pastro}
\maketitle
\begin{abstract}
We develop an alternative approach to star-autonomous comonads via linearly
distributive categories. It is shown that in the autonomous case the notions of
star-autonomous comonad and Hopf comonad coincide.
\end{abstract}

\section{Introduction}


Given a linearly distributive category $\C$, this note determines what
structure is required of a comonad $G$ on $\C$ so that $\C^G$, the category of
Eilenberg-Moore coalgebras of $G$, is again a linearly distributive category.
Furthermore, if $\C$ is equipped with negations (and is hence a star-autonomous
category), the structure required to lift the negations to $\C^G$ is determined
as well. This latter is equivalent to lifting star-autonomy and it is shown
that the notion presented is equivalent to a star-autonomous comonad~\cite{PS}.
As a consequence of the presentation given here, it may be easily seen that any
star-autonomous comonad on an autonomous category is a Hopf monad~\cite{BV}.

\section{Lifting linear distributivity}

Suppose $\C$ is a monoidal category and $G:\C \ra \C$ is a comonad on $\C$.
Recall that $\C^G$, the category of (Eilenberg-Moore) coalgebras of
$G$, is monoidal if and only if $G$ is a monoidal comonad~\cite{M}. In this
section we are interested in the structure required to lift linear
distributivity to the category of coalgebras.

A linearly distributive category $\C$ is a category equipped with two monoidal
structures $(\C,\star,I)$ and $(\C,\dia,J)$,\footnote{
    For simplicity we assume that the monoidal structures are
    strict, although this is not necessary. Furthermore, in their original
    paper~\cite{CS} the tensor products $\star$ and $\dia$ are respectively
    denoted by $\otimes$ and $\dia$, and called \emph{tensor} and \emph{par},
    emphasizing their connection to linear logic.} 
and two compatibility natural transformations (called ``linear distributions'')
\begin{align*}
    \partial_l &: A \star (B \dia C) \ra (A \star B) \dia C \\
    \partial_r &: (B \dia C) \star A \ra B \dia (C \star A),
\end{align*}
satisfying a large number of coherence diagrams~\cite{CS}.

Suppose $G = (G,\delta,\epsilon)$ is a comonad on a linearly distributive
category $\C$ which is a monoidal comonad on $\C$ with respect to both $\star$
and $\dia$, with structure maps $(G,\phi,\phi_0)$ and $(G,\psi,\psi_0)$
respectively. If, for $G$-coalgebras $A$, $B$, and $C$, the comonad $G$
satisfies
\begin{equation}\label{L1}
    \vcenter{\xygraph{{GA \star (GB \dia GC)}="1" [r(3)]
        {GA \star G(B \dia C)}="2" [r(2.7)]
        {G(A \star (B \dia C))}="3"
        "1"[d]{(GA \star GB) \dia GC}="4"
        "2"[d]{G(A \star B) \dia GC}="5"
        "3"[d]{G((A \star B) \dia C),}="6"
        "1":"2" ^-{1 \star \psi}
        "2":"3" ^-{\phi}
        "4":"5" ^-{\phi \dia 1}
        "5":"6" ^-{\psi}
        "3":"6" ^-{\partial_l}
        "1":"4" _-{\partial_l}}}
\end{equation}
it may be seen that the morphism $\partial_l$ becomes a $G$-coalgebra morphism.
If $G$ satisfies a similar axiom for $\partial_r$, i.e.,
\begin{equation}\label{L2}
    \vcenter{\xygraph{{(GB \dia GC) \star GA}="1" [r(3)]
        {G(B \dia C) \star GA}="2" [r(2.7)]
        {G((B \dia C) \star A)}="3"
        "1"[d]{GB \dia (GC \star GA)}="4"
        "2"[d]{GB \dia G(C \star A)}="5"
        "3"[d]{G(B \dia (C \star A)),}="6"
        "1":"2" ^-{\psi \star 1}
        "2":"3" ^-{\phi}
        "4":"5" ^-{1 \dia \phi}
        "5":"6" ^-{\psi}
        "3":"6" ^-{\partial_r}
        "1":"4" _-{\partial_r}}}
\end{equation}
then $\partial_r$ also becomes a $G$-coalgebra morphism. Thus,

\begin{proposition}
Given a linearly distributive category $\C$ and a comonad $G:\C \ra \C$
satisfying axioms~\eqref{L1} and~\eqref{L2}, the category $\C^G$ is a linearly
distributive category.
\end{proposition}

\begin{example}
Let $\C$ be a symmetric linearly distributive category and
$(B,\mu,\eta,\delta,\epsilon)$ a bialgebra in $\C$ with respect to $\dia$.
That is, the structure morphisms are given as
\begin{align*}
    \mu &: B \dia B \ra B  & \delta &: B \ra B \dia B \\
    \eta &: J \ra B        & \epsilon &: B \ra J.
\end{align*}
Then, $G = B \dia -$ is a comonad and is monoidal with respect to both $\star$
and $\dia$. The latter by $I \cong J \dia I \xra{\eta \dia 1} B * I$,
and the following,
\begin{align*}
    (B \dia U) \star (B \dia V)
        \xra{\makebox[8ex][c]{\scriptsize $\partial_r$}} &~
            B \dia (U \star (B \dia V)) \\
        \xra{\makebox[8ex][c]{\scriptsize $1 \dia (1 \star c)$}} &~
            B \dia (U \star (V \dia B)) \\
        \xra{\makebox[8ex][c]{\scriptsize $1 \dia \partial_l$}} &~
            B \dia ((U \star V) \dia B) \\
        \xra{\makebox[8ex][c]{\scriptsize $1 \dia c$}} &~
            B \dia (B \dia (U \star V)) \\
        \xra{\makebox[8ex][c]{\scriptsize $\cong$}} &~
            (B \dia B) \dia (U \star V) \\
        \xra{\makebox[8ex][c]{\scriptsize $\mu \star 1$}} &~ B \dia (U \star V).
\end{align*}
Rather large diagrams, which we leave to the faith of the reader, prove that $B
\dia -$ satisfies~\eqref{L1} and~\eqref{L2}, so that $\C^B =
\mathbf{Comod}_\C(B)$, the category of comodules of $B$, is a linearly
distributive category.
\end{example}

\section{Lifting negations}

Suppose now that $\C$ is a linearly distributive category equipped with
negations $S$ and $S'$ (corresponding to ${}^\bot(-)$ and $(-)^\bot$
in~\cite{CS}). That is, functors $S,S': \C^\op \ra \C$ together with the
following (dinatural) evaluation and coevaluation morphisms
\begin{equation}\label{negations}
\begin{array}{rl}
    SA \star A \xra{e_A} J   \qquad & \qquad  A \star S'A \xra{e'_A} J \\[1ex] 
    I \xra{n_A} A \dia SA  \qquad & \qquad  I \xra{n'_A} S'A \dia A,
\end{array}
\end{equation}
satisfying the four evident ``triangle identities''. One such is
\[
    \left(A \cong I \star A \xra{n \star 1} (A \dia SA) \star A \xra{\partial_r}
    A \dia (SA \star A) \xra{1 \dia e} A \dia J \cong A\right) = 1_A.
\]
If $\C$ is equipped with such negations we say simply that $\C$ is a
\emph{linearly distributive category with negations}.

We are interested to lift negations to $\C^G$. This means we must ensure that
the ``negation'' functors $S,S':\C^\op \ra \C$ lift to functors $(\C^G)^\op
\ra \C^G$, and the evaluation and coevaluation morphisms are in $\C^G$, i.e.,
are $G$-coalgebra morphisms.

The following is essentially known from~\cite{S}.

\begin{lemma}
A (contravariant) functor $S:\C^\op \ra \C$ may be lifted to a functor
$\wt{S}:(\C^G)^\op \ra \C^G$ such that the diagram
\[
    \xygraph{{(\C^G)^\op}="1" [r(1.5)] {\C^G}="2"
        "1"[d]{\C^\op}="3" "2"[d]{\C,}="4"
        "1":"2" ^-{\wt{S}}
        "2":"4" ^-{U}
        "1":"3" _-{U}
        "3":"4" ^-{S}}
\]
commutes, if and only if there is a natural transformation
\[
    \nu : S \ra GSG
\]
satisfying the following two axioms
\begin{equation}\label{lift-S}
    \vcenter{\xygraph{{S}="1" [r(1.6)]{GSG}="2" [d]{SG}="3"
        "1":"2" ^-{\nu}
        "2":"3" ^-{\epsilon_{SG}}
        "1":"3" _-{S\epsilon}}}
    \qquad
    \vcenter{\xygraph{{S}="1"
        [r(1.6)]{GSG}="2"
        [r(2.1)]{G^2SG}="3"
        "2"[d]{GSG}="4"
        "3"[d]{G^2SG^2.}="5"
        "1":"2" ^-{\nu}
        "2":"3" ^-{\delta_{SG}}
        "1":"4" _-{\nu}
        "4":"5" ^-{G\nu_G}
        "5":"3" _-{G^2S\delta}}}
\end{equation}
\end{lemma}
This may be viewed as a distributive law of a contravariant functor over a
comonad~\cite{S}. In this case, we say that \emph{$S$ may be lifted to $\C^G$},
and a functor $\wt{S}:(\C^G)^\op \ra \C^G$ is defined as
\[
    \wt{S}(A,\gamma) = \big(SA,\: SA \xra{\nu} GSGA \xra{GS\gamma} GA\big)
    \quad\qquad
    \wt{S}(f) = Sf.
\]
(To see the reverse direction, suppose $(A,\gamma)$ is a coalgebra and
$\wt{S}$ is a functor $\C^G \ra \C^G$, so that $\wt{S}A = (SA,\wt{\gamma})$ is
again a coalgebra. Define
\[
    \nu := SA \xra{\wt{\gamma}} GSA \xra{GS\epsilon_A} GSGA,
\]
which may be seen to satisfy the axioms in \eqref{lift-S}.)
We will usually let the context differentiate between $S$ and $\wt{S}$ and
simply write $S$ in both cases.

Now, suppose $S$ and $S'$ are equipped with natural transformations
\[
    \nu:S \ra GSG \qquad \text{and} \qquad \nu':S' \ra GS'G.
\]
such that they can be lifted to $\C^G$. It remains to lift the evaluation and
coevaluation morphisms~\eqref{negations}. Consider the following axioms.

\begin{equation}\label{Le}
    \vcenter{\xygraph{{SA \star GA}="1"
        [r(2.5)]{SA \star A}="2"
        [r(2.5)]{J}="3"
        "1"[d]{GSGA \star G^2A}="4"
        [r(2.7)]{G(SGA \star GA)}="5"
        [r(2.3)]{GJ}="6"
        "1":"4" _-{\nu \star \delta}
        "4":"5" ^-{\phi}
        "5":"6" ^-{Ge_{GA}}
        "1":"2" ^-{1 \star \epsilon}
        "2":"3" ^-{e_A}
        "3":"6" ^-{\psi_0}}}
\end{equation}
\begin{equation}\label{Ln}
    \vcenter{\xygraph{{I}="1"
        [d]{GA \dia SGA}="2"
        [r(3)]{GA \dia GSG^2A}="3"
        [r(3)]{G(A \dia SG^2A)}="4"
        [u]{G(A \dia SGA)}="5"
        "1"[r(1)]{GI}="6"
        [r(2)]{G(A \dia SA)}="7"
        "1":"2" _-{n}
        "2":"3" ^-{1 \dia \nu}
        "3":"4" ^-{\phi}
        "4":"5" _-{G(1 \dia S\delta)}
        "1":"6" ^-{\phi_0}
        "6":"7" ^-{Gn}
        "7":"5" ^-{G(1 \dia S\epsilon)}}}
\end{equation}
\begin{equation}\label{Le'}
    \vcenter{\xygraph{{GA \star S'A}="1"
        [r(2.5)]{A \star S'A}="2"
        [r(2.5)]{J}="3"
        "1"[d]{G^2A \star GS'GA}="4"
        [r(2.7)]{G(GA \star S'GA)}="5"
        [r(2.3)]{GJ}="6"
        "1":"4" _-{\delta \star \nu'}
        "4":"5" ^-{\phi}
        "5":"6" ^-{Ge'_{GA}}
        "1":"2" ^-{\epsilon \star 1}
        "2":"3" ^-{e'_A}
        "3":"6" ^-{\psi_0}}}
\end{equation}
\begin{equation}\label{Ln'}
    \vcenter{\xygraph{{I}="1"
        [d]{S'GA \dia GA}="2"
        [r(3)]{GS'G^2A \dia GA}="3"
        [r(3)]{G(S'G^2A \dia A)}="4"
        [u]{G(S'GA \dia A)}="5"
        "1"[r(1)]{GI}="6"
        [r(2)]{G(S'A \dia A)}="7"
        "1":"2" _-{n'}
        "2":"3" ^-{\nu' \dia 1}
        "3":"4" ^-{\phi}
        "4":"5" _-{G(S'\delta \dia 1)}
        "1":"6" ^-{\phi_0}
        "6":"7" ^-{Gn'}
        "7":"5" ^-{G(S'\epsilon \dia 1)}}}
\end{equation}

\begin{proposition}\label{prop-lift-neg}
Suppose $\C$ is a linearly distributive category with negation, $G$ is a
monoidal comonad satisfying axioms~\eqref{L1} and~\eqref{L2} (so that $\C^G$ is
linearly distributive), and that $S$ and $S'$ may be lifted to $\C^G$. Then,
$G$ satisfies axioms~\eqref{Le}, \eqref{Ln}, \eqref{Le'}, and~\eqref{Ln'} if
and only if $\C^G$ is a linearly distributive category with negation.
\end{proposition}

\begin{proof}
Suppose $(A,\gamma)$ is a $G$-coalgebra. We start by proving that
axiom~\eqref{Le} holds if and only if $e:SA \star A \ra J$ is a $G$-coalgebra
morphism. The following diagram proves the ``only if'' direction,
\[
    \xygraph{{SA \star A}="1"
        [r(2.5)]{GSGA \star GA}="2"
        [r(2.5)]{G(SGA \star A)}="3"
        [r(2)d]{G(SA \star A)}="4"
        "1"[d]{SA \star GA}="5"
        "2"[d]{GSGA \star G^2A}="6"
        "3"[d]{G(SGA \star GA)}="7"
        "4"[d]{GJ,}="8"
        "5"[d]{SA \star A}="9"
        "9"[r(3.5)]{J}="10"
        "9"[r(3)u(0.5)]{\txt{\footnotesize{\eqref{Le}}}}
        "1":@<-8pt>@/_20pt/"9" _-1
        "1":"2" ^-{\nu \star \gamma}
        "2":"3" ^-{\phi}
        "3":"4" ^-{G(S\gamma \star 1)}
        "1":"5" ^-{1 \star \gamma}
        "2":"6" _-{1 \star G\gamma}
        "3":"7" _-{G(1 \star \gamma)}
        "4":"8" ^-{Ge}
        "5":"6" ^-{\nu \star \delta}
        "6":"7" ^-{\phi}
        "7":"8" ^-{Ge}
        "5":"9" ^-{1 \star \epsilon}
        "9":"10" ^-{e}
        "10":"8" ^-{\psi_0}}
\]
and this next diagram the ``if'' direction
\[
    \xygraph{{SA \star GA}="1"
        [r(3.75)]{GSGA \star G^2A}="2"
        "1"[l(1.5)d]{SA \star A}="3"
        "1"[d]{SGA \star GA}="4"
        [r(2.4)]{GSG^2A \star G^2A}="5"
        [r(2.7)]{GSGA \star G^2A}="6"
        [r(2.3)]{G(SGA \star GA)}="7"
        "4"[d]{J}="8"
        "7"[d]{GJ,}="9"
        "1":"3" _-{1 \star \epsilon}
        "1":"4" ^-{S\epsilon \star 1}
        "1":"2" ^-{\nu \star \delta}
        "2":"5" _-{GSG\epsilon \star 1}
        "2":"6" ^-{1}
        "2":@/^3ex/"7" ^-{\phi}
        "3":"8" _-{e}
        "4":"8" ^-{e}
        "4":"5" ^-{\nu \star \delta}
        "5":"6" ^-{GS\delta \star 1}
        "6":"7" ^-{\phi}
        "7":"9" ^-{Ge}
        "8":"9" ^-{\psi_0}
        }
\]
where the bottom square commutes as $e_{GA}$ is a $G$-coalgebra morphism.

Next we prove that axiom~\eqref{Ln} holds if and only if $n:I
\ra A \dia SA$ is a $G$-coalgebra morphism. The ``only if'' direction is given
by
\[
    \xygraph{{I}="1"
        [d]{A \dia SA}="4"
        [r(1.5)]{GA \dia SGA}="5"
        [r(2.1)]{GA \dia GSG^2A}="6"
        [r(2.3)]{G(A \dia SG^2A)}="7"
        [r(2.7)]{G(A \dia SGA)}="8"
        "1"[r(4.3)]{GI}="2"
        "5"[d]{GA \dia SA}="9"
        "6"[d]{GA \dia GSGA}="10"
        "7"[d]{G(A \dia SGA)}="11"
        "8"[d]{G(A \dia SA),}="12"
        "8"[u]{G(A \dia SA)}="3"
        "3"[d(0.4)l(3.6)]{\txt{\footnotesize{\eqref{Ln}}}}
        "1":"2" ^-{\phi_0}
        "2":"3" ^-{Gn}
        "3":"8" _-{G(1 \dia S\epsilon)}
        "8":"12" _-{G(1 \dia S\gamma)}
        "1":"5" ^-n
        "5":"6" ^-{1 \dia \nu}
        "6":"7" ^-{\phi}
        "7":"8" ^-{G(1 \dia S\delta)}
        "1":"4" _-n
        "4":"9" _-{\gamma \dia 1}
        "9":"10" ^-{1 \dia \nu}
        "10":"11" ^-{\phi}
        "11":"12" ^-{G(1 \dia S\gamma}
        "5":"9" ^-{1 \dia S\gamma}
        "6":"10" _-{1 \dia GSG\gamma}
        "7":"11" _-{G(1 \dia SG\gamma)}
        "3":@<15pt>@/^20pt/"12" ^-1}
\]
and the ``if'' direction by
\[
    \xygraph{{I}="1"
        [d]{GA \dia SGA}="3"
        [r(2.3)]{G^2A \dia GSG^2A}="4"
        [r(2.5)]{G(GA \dia SG^2A)}="5"
        [r(2.9)]{G(GA \dia SGA)}="6"
        [r(1.8)]{G(A \dia SA)}="7"
        "6"[u]{GI}="2"
        "4"[d]{GA \dia GSG^2A}="8"
        "5"[d]{G(A \dia SG^2A)}="9"
        "6"[d]{G(A \dia SGA),}="10"
        "1":"2" ^-{\phi_0}
        "1":"3" _-{n}
        "2":"6" _-{Gn}
        "2":"7" ^-{Gn}
        "3":"4" ^-{\delta \dia \nu}
        "3":"8" _-{1 \dia \nu}
        "4":"5" ^-{\psi}
        "4":"8" ^-{G\epsilon \dia 1}
        "5":"6" ^-{G(1 \dia S\delta)}
        "5":"9" ^-{G(\epsilon \dia 1)}
        "6":"10" _-{G(\epsilon \dia 1)}
        "7":"10" ^-{G(1 \dia G\epsilon)}
        "8":"9" ^-{\psi}
        "9":"10" ^-{G(1 \dia S\delta)}}
\]
where the top square commutes as $n_{GA}$ is a $G$-coalgebra morphism.

The remaining two axioms are proved similarly.
\end{proof}

\section{Star-autonomous comonads}

Suppose $\C = (\C,\ox,I)$ is a star-autonomous category. A star-autonomous
comonad $G:\C \ra \C$ is a comonad satisfying axioms (described below) so that
$\C^G$ becomes a star-autonomous category~\cite{PS}. In this section we show
that comonads as in Proposition~\ref{prop-lift-neg} and star-autonomous
comonads coincide.

We recall the definition of star-autonomous comonad~\cite{PS}, but, as it
suits our needs better here, we present a more symmetric version. First recall
that a star-autonomous category may be defined as a monoidal category $\C =
(\C,\ox,I)$ equipped with an equivalence
\[
    S \dashv S' : \C^\op \ra \C
\]
such that
\begin{equation}\label{star-auto-iso}
    \C(A \ox B, SC) \cong \C(A, S(B \ox C)),
\end{equation}
natural in $A,B,C \in \C$. The functor $S$ is called the \emph{left star
operation} and $S'$ the \emph{right star operation}.

By the Yoneda lemma, the isomorphism in~\eqref{star-auto-iso} determines, and
is determined by, the two following ``evaluation'' morphisms:
\[
    e = e_{A,B} : S(A \ox B) \ox A \ra SB
    \quad \text{and} \quad
    e' = e'_{B,A} : B \ox S'(A \ox B) \ra S'A.
\]

\begin{definition}
A \emph{star-autonomous comonad} on a star-autonomous category $\C$ is a
monoidal comonad $G:\C \ra \C$ equipped with
\[
    \nu:S \ra GSG \qquad \text{and} \qquad \nu':S' \ra GS'G,
\]
satisfying~\eqref{lift-S} (i.e., $S,S'$ may be lifted to $\C^G$), and this
data must be such that the following four diagrams commute.
\[
    \xygraph{{SS'G}="1" [r(2)]{G}="2" "1"[d]{GSGS'G}="3" "2"[d]{GSS'}="4"
        "1":"2" ^-\cong
        "2":"4" ^-\cong
        "1":"3"_-{\nu}
        "3":"4"^-{GS\nu'}}
    \qquad\qquad
    \xygraph{{S'SG}="1" [r(2)]{G}="2" "1"[d]{GS'GSG}="3" "2"[d]{GS'S}="4"
        "1":"2" ^-\cong
        "2":"4" ^-\cong
        "1":"3"_-{\nu'}
        "3":"4"^-{GS\nu}}
\]
\[
\vcenter{\xygraph{{S(A \ox B) \ox GA}="1"
    [r(3)]{S(A \ox B) \ox A}="2"
    [r(2.2)]{SB}="3"
    "1"[l(0.4)d]{GSG(A \ox B) \ox G^2A}="4"
    "3"[r(0.4)d]{GSGB}="5"
    "4"[r(0.4)d]{G(SG(A \ox B) \ox GA)}="6"
    [r(5.2)]{G(S(GA \ox GB) \ox GA)}="7"
    "1":"2" ^-{1 \ox \epsilon}
    "2":"3" ^-{e_{A,B}}
    "3":"5" ^-{\nu}
    "1":"4" _-{\nu \ox \delta}
    "4":"6" _-{\phi}
    "6":"7" _-{G(S\phi \ox 1)}
    "7":"5" _-{Ge_{GA,GB}}}}
\]
\[
\vcenter{\xygraph{{GB \ox S'(A \ox B)}="1"
    [r(3)]{B \ox S'(A \ox B)}="2"
    [r(2.2)]{S'A}="3"
    "1"[l(0.4)d]{G^2B \ox GS'G(A \ox B)}="4"
    "3"[r(0.4)d]{GS'GA}="5"
    "4"[r(0.4)d]{G(GB \ox S'G(A \ox B))}="6"
    [r(5.2)]{G(GB \ox S'(GA \ox GB))}="7"
    "1":"2" ^-{\epsilon \ox 1}
    "2":"3" ^-{e'_{B,A}}
    "3":"5" ^-{\nu'}
    "1":"4" _-{\delta \ox \nu'}
    "4":"6" _-{\phi}
    "6":"7" _-{G(1 \ox S'\phi)}
    "7":"5" _-{Ge'_{GB,GA}}}}
\]
\end{definition}
The first two diagrams above ensure that the equivalence $S \simeq S'$ lifts to
$\C^G$, while the latter two diagrams above respectively ensure that $e$ and
$e'$ are $G$-coalgebra morphisms, so that the isomorphism~\eqref{star-auto-iso}
also lifts to $\C^G$.

We wish to show that star-autonomous comonads and comonads as in
Proposition~\ref{prop-lift-neg} coincide. It should not be surprising given the
following theorem.

\begin{theorem}[{\cite[Theorem~4.5]{CS}}]\label{star-LDC-same}
The notions of linearly distributive categories with negation and
star-autonomous categories coincide.
\end{theorem}

Given a star-autonomous category, identifying $\star := \ox$ (and the units
$I := I_\star = I_\ox$) and defining
\begin{equation}\label{star-LDC}
A \dia B := S'(SB \star SA) \cong S(S'B \star S'A) 
\qquad
J := SI \cong S'I
\end{equation}
gives a linearly distributive category~\cite{CS}. The negations of course come
from $S$ and $S'$. In~\cite{CS}, they consider the symmetric case, but the
correspondence between linearly distributive categories with negation and
star-autonomous categories holds in the noncommutative case as well.

Now, given Theorem~\ref{star-LDC-same}, Proposition~\ref{prop-lift-neg} says
that if $\C$ is star-autonomous, and $G$ is such a comonad, then $\C^G$ is
star-autonomous. We now compare the two definitions.

Suppose now that $G$ is a comonad on a linear distributive category $\C$ as in
Proposition~\ref{prop-lift-neg}. We
wish to show that it is a star-autonomous comonad. Rather than proving the
axioms, it is simpler to show directly that the morphisms under consideration
are $G$-coalgebra morphisms. To this end, the equivalence $S \simeq S'$ is
given by the equations
\[
    A \cong I \star A \xra{n'_{SA} \star 1} (S'SA \dia SA) \star A
    \xra{\partial_r} S'SA \dia (SA \star A) \xra{1 \dia n} S'SA \dia J
    \cong S'SA
\]
and
\[
    S'SA \cong I \star S'SA \xra{n_A \star 1} (A \dia SA) \star S'SA
    \xra{\partial_r} A \dia (SA \star S'SA) \xra{1 \dia e'_{SA}} A \dia J
    \cong A,
\]
and $e_{A,B}$ and $e'_{B,A}$ are respectively defined as
\[
    \xygraph{{S(A \star B) \star A}="1"
        [d]{S(A \star B) \star A \star I}="3"
        [d]{S(A \star B) \star A \star (B \dia SB)}="5"
        [r(5)]{(S(A \star B) \star A \star B) \dia SB}="6"
        [u]{J \dia SB}="4"
        [u]{SB}="2"
        "1":"3"_-\cong
        "3":"5"_-{1 \star 1 \star n}
        "5":"6"^-{\partial_l}
        "6":"4"_-{e_{A \star B} \dia 1}
        "4":"2"_-\cong
        "1":@{.>}"2"^-{e_{A,B}}}
\]
\[
    \xygraph{{B \star S'(A \star B)}="1"
        [d]{I \star B \star S'(A \star B)}="3"
        [d]{(S'A \dia A) \star B \star S'(A \star B)}="5"
        [r(5)]{S'A \dia (A \star B \star S'(A \star B))}="6"
        [u]{S'A \dia J}="4"
        [u]{SB}="2"
        "1":"3"_-\cong
        "3":"5"_-{n' \star 1 \star 1}
        "5":"6"^-{\partial_r}
        "6":"4"_-{1 \dia e'_{A \star B}}
        "4":"2"_-\cong
        "1":@{.>}"2"^-{e'_{B,A}}}
\]
In the situation of Proposition~\ref{prop-lift-neg}, we see that all four of
these morphisms are given as composites of $G$-coalgebra morphisms, and thus,
are $G$-coalgebra morphisms themselves. Therefore, $G$ is a star-autonomous
comonad. 

In the other direction suppose $G$ is a star-autonomous comonad on a
star-autonomous category $\C$. It is similar to show that it is a comonad
satisfying the requirements of Proposition~\ref{prop-lift-neg}. Using the
identifications in~\eqref{star-LDC}, the two linear distributions are defined
as follows.
\[
    \xygraph{{A \star (B \dia C)}="1"
        [d]{A \ox S'(SC \ox SB)}="3"
        [r(1.3)d]{A \ox S'(SC \ox S(A \ox B) \ox A)}="5"
        [r(1.3)u]{S'(SC \ox S(A \ox B))}="4"
        [u]{(A \star B) \dia C}="2"
        "1":@{.>}"2" ^-{\partial_l}
        "1":"3" _-\cong
        "3":"5" _<(0.2){1 \ox S'(1 \ox e)}
        "5":"4" _<(0.5){e'}
        "4":"2" _-\cong}
\quad
    \xygraph{{(B \dia C) \star A}="1"
        [d]{S(S'C \ox S'B) \ox A}="3"
        [r(1.3)d]{S(A \ox S'(C \ox A) \ox S'B) \ox A)}="5"
        [r(1.3)u]{S(S'(C \ox A) \ox S'B)}="4"
        [u]{B \dia (C \star A)}="2"
        "1":@{.>}"2" ^-{\partial_r}
        "1":"3" _-\cong
        "3":"5" _<(0.2){S(e' \ox 1) \ox 1}
        "5":"4" _<(0.4){e}
        "4":"2" _-\cong}
\]
The evaluation maps $e_A$ and $e'_A$ are defined as $e_{A,I}$ and $e'_{A,I}$,
and the coevaluation maps $n_A$ and $n'_A$ as
\[
    n_A = \Big(I \cong SS'I \xra{Se'_{A,I}} S(A \ox S'A) = A \dia SA\Big)
\]
\[
    n'_A = \Big(I \cong S'SI \xra{S'e_{A,I}} S'(SA \ox A) = S'A \dia A\Big)
\]
Again, each morphism is a $G$-coalgebra morphism, or composite thereof, and
therefore is itself a $G$-coalgebra morphism.

Thus, both notions coincide, and we will simply call either notion a
\emph{star-autonomous comonad}, and let context differentiate the
axiomatization.

\begin{example}
Any Hopf algebra $H$ in a star-autonomous category $\C$ gives rise to a
star-autonomous comonad $H \ox - : \C \ra \C$. See~\cite[pg. 3515]{PS} for details.
\end{example}

\begin{example}
If $\C$ is a symmetric closed monoidal category with finite products, then we
may apply the Chu construction~\cite{B} to produce a star-autonomous category
$\Chu(\C)$. $\C$ fully faithfully embeds into $\Chu(\C)$,
\[
    \C \hookrightarrow \Chu(\C)
\]
and this functor is strong symmetric monoidal. Thus, any Hopf algebra in $\C$
becomes a Hopf algebra in $\Chu(\C)$, and thus, an example of a star-autonomous
comonad.
\end{example}

\section{The compact case $\star = \dia$}

If $\C$ is a linearly distributive category with negation for which $\star =
\dia$ (and thus, $I = J$), then $\C$ is an autonomous (= rigid) category.
The functor $S$ provides left duals, while $S'$ provides right duals. It is
not hard to see that in this case, any star-autonomous monad $G$ (after
dualizing) is a Hopf monad~\cite{BV}. Set $\star = \dia$ and $I = J$ and
dualize axioms~\eqref{Le},~\eqref{Ln},~\eqref{Le'}, and~\eqref{Ln'}. They
correspond in~\cite{BV} to axioms (23), (22), (21), and~(20) respectively.
(In their notation ${}^\vee(-) = S$ and $(-)^\vee = S'$.) Therefore, we have:

\begin{proposition}
Star-autonomous monads on autonomous categories are Hopf monads.
\end{proposition}



\subsubsection*{Acknowledgements}

I would like to thank Robin Cockett and Masahito Hasegawa for their valuable
suggestions.


\bigskip
{\small\noindent
Department of Mathematics, Kyushu University, 744 Motooka, Nishi-ku,
Fukuoka 819-0395, Japan \\
\texttt{craig@math.kyushu-u.ac.jp}
}

\begin{thebibliography}{8}

\bibitem[B79]{B}
Michael Barr. $*$-Autonomous categories, Volume 752 of Lecture Notes in
Mathematics. Springer, Berlin, 1979. With an appendix by Po Hsiang Chu.

\bibitem[BV07]{BV}
Alain Brugui\`eres and Alexis Virelizier. Hopf monads, Advances in
Mathematics 215 no. 2 (2007) 679--733.

\bibitem[CS97]{CS}
J.R.B. Cockett and R.A.G. Seely. Weakly distributive categories, Journal of
Pure and Applied Algebra 114 (1997) 133--173. Corrected version available from 
the second authors webpage.

\bibitem[M02]{M}
I. Moerdijk. Monads on tensor categories, Journal of Pure and Applied Algebra
168 (2002) 189--208. 

\bibitem[PS09]{PS}
Craig Pastro and Ross Street. Closed categories, star-autonomy, and
monoidal comonads, Journal of Algebra 321 no. 11 (2009) 3494--3520.

\bibitem[S72]{S}
Ross Street. The formal theory of monads, Journal of Pure and Applied Algebra
2 no. 2 (1972) 149--168.

\end{thebibliography}
\end{document}